\documentclass[openacc]{rsproca_new}

\usepackage{amsthm,amsmath,amssymb}
\usepackage{graphicx,color,epsfig}

\usepackage{changes}

\def\C{\mathbb{C}}
\def\N{\mathbb{N}}
\def\R{\mathbb{R}}
\def\T{\mathbb{T}}
\def\Z{\mathbb{Z}}

\newtheorem{Thm}{\bf Theorem}[section]
\newtheorem{Pro}[Thm]{\bf Proposition}
\newtheorem{Lem}[Thm]{\bf Lemma}

\begin{document}
\title{The mathematics of asymptotic stability in the Kuramoto model}
\author{Helge Dietert$^1$ and Bastien Fernandez$^2$}

\address{$^1$ Institut de Math\'ematiques de Jussieu - Paris Rive Gauche, Universit\'e Paris 7 Denis Diderot - Sorbonne Paris Cit\'e, 75205 Paris CEDEX 13 France\\
$^2$ Laboratoire de Probabilit\'es, Statistique et Mod\'elisation, CNRS - Universit\'e Paris 7 Denis Diderot - Sorbonne Universit\'e, 75205 Paris CEDEX 13 France}

\corres{Bastien Fernandez\\
\email{fernandez@lpsm.paris}}

\begin{abstract}
Now a standard in Nonlinear Sciences, the Kuramoto model is the perfect example of the transition to synchrony in heterogeneous systems of coupled oscillators. While its basic phenomenology has been sketched in early works, the corresponding rigorous validation has long remained problematic and was achieved only recently. This paper reviews the mathematical results on asymptotic stability of stationary solutions in the continuum limit of the Kuramoto model, and provides insights into the principal arguments of proofs. This review is complemented with additional original results, various examples, and possible extensions to some variations of the model in the literature.
\end{abstract}

\maketitle

\section{Introduction}
\subsection{The Kuramoto model of coupled oscillators}
The Kuramoto model is the archetype of collective systems composed of heterogeneous individuals that are influenced by attractive pairwise interactions. Originally designed to mimic chemical instabilities \cite{K75,K84}, it has since become a standard of the transition to synchrony in agent-based systems, and has been applied to various examples in disciplines such as Condensed Matter, Neuroscience and Humanities \cite{ABP-VRS05,PRK01}.

In its simplest form, this model considers a collection of $N\in\N$ oscillators, represented by their phase $\theta_i$, a variable in the unit circle $\T^1=\R/2\pi\Z$. The population dynamics is governed by the following set of globally coupled first order ODEs
\begin{equation}
\frac{d\theta_i}{dt}=\omega_i+\frac{K}{N}\sum_{j=1}^N\sin (\theta_j-\theta_i),\quad \forall i\in\{1,\dots,N\}.
\label{ORIKURA}
\end{equation}
The time-independent frequencies $\omega_i\in\R$ are randomly drawn in order to account for individual heterogeneities. The parameter $K\in\R^+$ measures the interaction strength. While, up to time rescaling, $K$ could be absorbed in the frequencies $\omega_i\mapsto \omega_i/K$, it is more convenient to investigate the dependence of the dynamics upon this parameter, for a given frequency distribution.

Clever intuition, elaborate analytic considerations and extensive numerics have provided comprehensive insights into the Kuramoto phenomenology, see e.g.\ \cite{D89,E85,KN87,MS05,OA08,PM15,SMM92,vHW93}. Nonetheless, due to heterogeneities, statements about the full nonlinear dynamics, certified by complete mathematical proofs, are rather scarce. They can be summarized as follows.

For weak interactions, KAM theory for dissipative systems \cite[Thm 6.1]{BHS96} or \cite[Thm 3.1]{CL-HB05} asserts that, for frequencies in a Lebesgue positive set in $\R^N$, the dynamics for $K$ small is conjugated to the system at $K=0$.\footnote{For a nice presentation of the original KAM theory in the Hamiltonian context, see \cite{SD14}.} This conjugacy implies infinite returns to arbitrary small neighborhoods of the initial condition in $\T^N$.\footnote{An open problem is to evaluate the dependence on $N$ of the related estimates, see \cite{W84} for similar considerations in Hamiltonian chains of coupled oscillators.}

Results for strong interactions contrast with weak coupling recurrence, see \cite{BCM15,DB11,HHK10,VM08} and also \cite{BDP12,CHJK12,DB11} for additional interesting statements.
For $K>\max_{i,j}|\omega_i-\omega_j|$ and provided that initial phase spreading is limited enough, full locking asymptotically takes place; ie.\ the limit
\[
\lim_{t\to +\infty}|\theta_i(t)-\theta_j(t)|
\]
exists for every pair $(i,j)$. In the extreme case of homogeneous populations (ie.\ $\omega_i$ independent of $i$), complete synchrony holds for every $K>0$, viz.\ we have
\[
\lim_{t\to +\infty}\max_{i,j}|\theta_i(t)-\theta_j(t)|=0
\]
for almost every initial phases. However, for exceptional initial conditions, 
the population can cluster into two synchronized groups\cite{BCM15,CEM17}. This clustering is not limited to homogeneous populations and may hold for symmetric frequency distributions \cite{CP09}.

No rigorous results exist about \eqref{ORIKURA} in regimes when interactions and heterogeneities effects balance. However, insights can be obtained by considering the continuum limit approximation.

\subsection{The Kuramoto PDE: basic features}
\subsubsection{Kuramoto dynamics at the continuum limit}\label{S-CONT}
The continuum approximation assumes that populations at the thermodynamic limit $N\to +\infty$ are described by absolutely continuous distributions $f$ on the cylinder $\T^1\times\R$, more precisely, by their densities. Under this assumption, time evolution is governed by the following PDE \cite{S88,SM91}
\begin{equation}
\partial_t f + \partial_\theta (fV[f])= 0
\label{KPDE}
\end{equation}
where
\[
V[f](\theta,\omega)=\omega + K \int_{\T^1\times\R} \sin(\theta' - \theta) f(d\theta',d\omega'),\ \forall (\theta,\omega)\in\T^1\times\R.
\]
For any trajectory of \eqref{ORIKURA}, the empirical measure $\mu_N(t)=\frac1{N}\sum\limits_{i=1}^N\delta_{\theta_i(t),\omega_i}$ (here $\delta_{\cdot ,\cdot}$ stands for the Dirac distribution) is a weak solution of \eqref{KPDE}. The continuum approximation is justified by the following basic features \cite{L05}, which are common to mean-field models in classical mechanics, especially the Vlasov equation \cite{BH77,D79,G13}.
\begin{itemize}
\item[$\bullet$] The Cauchy problem is globally well-posed for \eqref{KPDE}, viz.\ for every initial probability measure $f(0)$ on the cylinder, there exists a unique solution $t\mapsto f(t)$ defined for all $t\geq 0$. If $f(0)$ is absolutely continuous, then so is $f(t)$ for every $t>0$.
\item[$\bullet$] The solution continuously depends on the initial condition, in the weak topology. More precisely, if $d_{\text BL}(\cdot,\cdot)$ denotes the bounded Lipschitz distance of measures, then there exists $C>0$ such that for every pair of solution $t\mapsto f_i(t)$, $i=1,2$, we have
\[
d_{\text BL}(f_1(t),f_2(t))\leq d_{\text BL}(f_1(0),f_2(0))\; e^{Ct},\quad \forall t>0.
\]
\end{itemize}
For $N$ sufficiently large, $\mu_N(0)$ can be chosen close to an absolutely continuous distribution $f(0)$, ie.\ such that $d_{\text BL}(\mu_N(0),f(0))$ is small. The inequality above then implies the continuum approximation on finite time interval, ie.\ $d_{\text BL}(\mu_N(t),f(t))$ remains small for $t$ small enough.

In addition, the Kuramoto PDE has the following specific features.
\begin{itemize}
\item[$\bullet$] Galilean invariance: if $t\mapsto f(t)$ is a solution, then $t\mapsto R_{\Theta+\Omega t,\Omega}f(t)$ is a solution for every $(\Theta,\Omega)\in \T^1\times\R$, where $R_{\Theta,\Omega}$ is the representation on measures of the map $(\theta,\omega)\mapsto (\theta+\Theta,\omega+\Omega)$.  In particular, \eqref{KPDE} is equivariant with respect to the rigid rotation $R_\Theta:=R_{\Theta,0}$.
\item[$\bullet$] For every solution $t\mapsto f(t)$, the frequency marginal $\int_{\T^1}f(t,d\theta,d\omega)$ does not depend on $t$, and thus can be regarded as an input parameter on initial conditions.
\end{itemize}

\subsubsection{Basic phenomenology}\label{S-BASPHE}
The degree of synchrony in Kuramoto dynamics can be characterized by the order parameter
\[
r(t)=\int_{\T^1\times\R}e^{i\theta}f(t,d\theta,d\omega).
\]
In particular, stationary states can be classified accordingly: the case $r=0$ corresponds to the homogeneous stationary state $f_\text{hom}(d\theta,d\omega)=\frac{g(\omega)}{2\pi}d\theta d\omega$ for which $\theta$ is uniformly distributed independently of $\omega$, while $r\neq 0$ deals with partially locked states (PLS), for which angles such that $|\theta |\leq |r|\omega$ are uniquely attached to a single value of $\omega$. (An expression of PLS is given below.)

The Kuramoto PDE displays a large phenomenology depending on the interaction strength and the (absolutely continuous) frequency marginal.  The simplest case is when the corresponding density $g$ is unimodal and symmetric. (Thanks to Galilean invariance, its maximum can be set at the origin $0$). Then, the phenomenology can be summarized as follows (see Figure \ref{BIFDIAG}, left) \cite{S00}:
\begin{figure}[h]
\begin{center}
\includegraphics[scale=0.9]{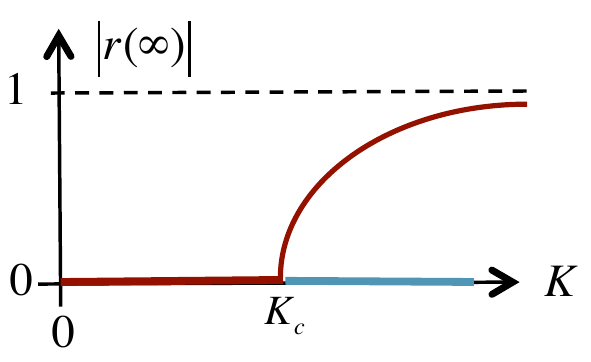}
\hspace{1cm}
\includegraphics[scale=0.9]{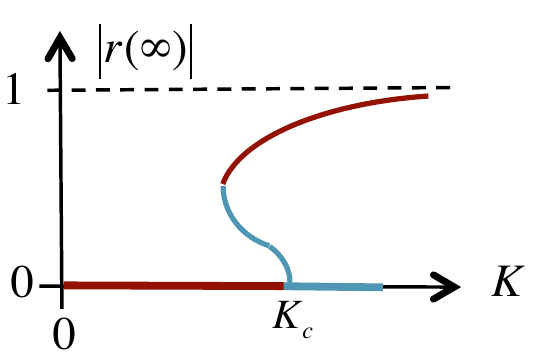}
\end{center}
\caption{Schematic bifurcation diagram (left) for a symmetric and unimodal frequency distribution $g$ and (right) for the Bi-Cauchy distribution $g_{\Delta,\Omega}$ when bimodal (see Section \ref{S-STABCOND}).  Red (resp.\ blue) lines indicate stable (resp.\ unstable) stationary solutions, see text for details.}
\label{BIFDIAG}
\end{figure}
\begin{itemize}
\item[$\bullet$] for $K<K_c:=\frac{2}{\pi g(0)}$, $f_\text{hom}$ is
  asymptotically stable. Hence, for every trajectory the order
    parameter asymptotically vanishes, $\lim_{t\to +\infty}r(t)=0$. This convergence
  is the analogue of the Landau damping phenomenon in the Vlasov
  equation \cite{V10}. (It is incompatible with KAM induced recurrence
  behaviors in \eqref{ORIKURA}. Hence, the continuum approximation
  mentioned above cannot hold for all $t>0$.)
\item[$\bullet$] at $K=K_c$, $f_\text{hom}$ becomes unstable and a circle of stable stationary PLS emerges for $K>K_c$ (together with a continuum of unstable PLS). In this regime, we have $\lim_{t\to +\infty}|r(t)|=|r_\text{pls}|\neq 0$ provided that $f(0)$ is not in the stable manifold of $f_\text{hom}$.
\end{itemize}
Illustrations of the dynamics in the finite dimensional model, both for $K<K_c$ and $K_>K_c$, are provided in the Supplementary Movies.
More elaborate bifurcation schemes occur for other marginals \cite{BU08,MBSOSA09}, especially for the bi-Cauchy distribution, see Fig.\ \ref{BIFDIAG} (right) and details in Section \ref{S-STABCOND}, and also for extensions of the model \cite{IPMS13,KP13,MP11,OW12,OW13}. Even for asymmetric unimodal marginals, the phenomenology can be involved.

\subsubsection{Proving the phenomenology: state-of-art and technical considerations}
While the phenomenology above had been identified in early studies, full rigorous confirmation has remained elusive until recently.
Mathematical studies have long been limited to linearized dynamics in strong topology. They have provided both stability criteria \cite{SM91} and evidence that the relaxation rate, either algebraic or exponential, depends on $g$'s regularity \cite{SMM92}. One should also mention that, for special frequency marginals, a rather impenetrable proof of $f_\text{hom}$ stability is exposed in \cite{C15}. In addition, complete synchronization has been proved to hold when $g$ is the Dirac distribution \cite{CCHKK14}. Besides, solid arguments have been provided for convergence of the order parameter dynamics to the corresponding one in the so-called Ott-Antonsen (OA) manifold \cite{OA09,OHA11}. There, the dynamics is governed by a finite-dimensional system when $g$ is meromorphic with finitely many poles in the lower half-plane \cite{OA08}; hence a standard analysis of stability and bifurcations can be developed in this case \cite{MBSOSA09} (see also Section 5.6.2 in \cite{D16b}).

The major obstacle to including nonlinearities in proofs is that, due to the free transport term $\omega\partial_\theta f$ in \eqref{KPDE}, in strong topology, the linearized dynamics has continuous spectrum on the imaginary axis \cite{MS07}. In fact, stationary states have all been shown to be nonlinearly unstable in the $L^2$-norm \cite{D16a,N15}. Therefore, any proof of asymptotic stability must consider weaker topology.

For suitable norms in weak topology and analytic frequency marginals, the linearized dynamics essential spectrum is located to the left of the imaginary axis in the complex plane \cite{DFG-V16}. Provided that the remaining discrete spectrum is under control -- hence the stability conditions -- a standard strategy for asymptotic stability can be considered: since the linearized dynamics decays exponentially fast, it can dominate nonlinear instabilities for small enough perturbations. In practice, the proof is not so straightforward and needs adjustments, especially because angular derivatives in \eqref{KPDE} imply that nonlinearities can be large even for small perturbations.

When the frequency marginal has only algebraic regularity, this strategy no longer applies because no spectral gap is at hand. Instead, the specific structure of linearized perturbation dynamics, which takes the form of a Volterra equation, needs to be exploited in order to prove algebraic damping via advanced bootstrap arguments \cite{D17}.

Besides, PLS stability deals with circles of stationary states, by equivariance with respect to rotations $R_\Theta$. The perturbation dynamics then must be neutral with respect to tangential perturbations. Asymptotic stability is to be proved for the relative equilibrium of the dynamics in the radial variable \cite{HI11}.

\subsection{Organization of the rest of the paper}
This paper aims to review stability results and their proofs for stationary solutions of \eqref{KPDE} that have been obtained in \cite{D16a,D16b,D17,DFG-V16,FG-VG16}. In few words, these results claim asymptotic convergence in the weak sense to either $f_\text{hom}$ or to some PLS, depending  on a corresponding stability condition. In addition, control of the order parameter relaxation speed will be given, which depends on the regularity of the initial condition (including the frequency distribution). Of note, thanks to Galilean invariance, all results immediately extend to globally rotating solutions.

The results are presented in Section \ref{S-RESULTS}. Section \ref{S-PROOF} provides insights into the main arguments of proofs, especially those that are likely to be of interest to readers not familiar with the analysis of PDEs. This includes linear stability analysis via considerations on Volterra equations and control of nonlinear terms by means of a Gearhart-Pr\"uss-like argument. The key point is to obtain weighted $L^2$ estimates on the solution's Fourier transform. This is not only critical for the proofs, but it also implies both convergence to the center manifold and to the OA manifold mentioned above (Section \ref{S-OA}).

Stability conditions in the statements will be expressed in terms of $K$ and $g$. Consistency considerations on these conditions are evaluated in Section \ref{S-STABCOND}, where bifurcation diagrams are also provided for various examples of frequency distributions, including original ones.

Finally, Section \ref{S-EXT} mentions some extensions of the Kuramoto model for which the approaches presented here apply to yield rigorous results on asymptotic stability. Limitations and open questions are also briefly discussed.

\section{Main results: Asymptotic stability in the Kuramoto PDE}\label{S-RESULTS}
This section describes the behavior of solutions $t\mapsto f(t)$ of \eqref{KPDE}, for a given absolutely continuous frequency marginal $\int_{\T^1}f(t,d\theta,d\omega)=g(\omega)d\omega$. More precisely, existence and stability conditions are given for the stationary states, together with the corresponding local basins of attraction.

The order parameter relaxation speed depends on the regularity of the
frequency marginal and of the initial
perturbation: more regularity implies faster decay. Regularity is usually quantified
by Fourier transform decay. Given a function $u$ on $\R$ and a measure
$v$ on $\T^1\times\R$, their Fourier transforms are defined by
\[
\widehat{u}(\tau)=\int_{\R}u(\omega) e^{-i\tau\omega} d\omega,\ \forall \tau\in\R \quad\text{and}\quad \widehat{v}_\ell(\tau)=\int_{\T^1\times\R}e^{-i(\ell\theta+\tau\omega)}v(d\theta,d\omega),\ \forall (\ell,\tau)\in\Z\times\R.
\]
Various constraints on Fourier transforms have been suggested, see end of Subsection \ref{S-HOMSTAB} below. For simplicity, we shall express constraints in terms of weighted norms.
Given a weight function $\phi:\R^+\to \R^+$ and a sequence of functions $u=\{u_\ell(\tau)\}\in \C^{\N\times\R^+}$, let
\[
\|u\|_{{\mathcal H}^1_\phi(\N\times\R^+)}=\left(\sum_{\ell\in\N}\int_{\R^+} \phi(\tau)^2\left(|u_\ell(\tau)|^2+|u'_\ell(\tau)|^2\right)d\tau\right)^{\frac12}
\]
and let also
\[
{\mathcal H}^1_\phi(\N\times\R^+)=\{u\in \C^{\N\times\R^+}\ :\ \|u\|_{{\mathcal H}^1_\phi(\N\times\R^+)}<+\infty\}.
\]
Typical weights are $\phi(\tau)= e^{a\tau}$ ($a>0$) and
$\phi(\tau)= (1+\tau)^b$ ($b>1$). In $\C^{\N\times \R}$, the norm
$\|\cdot \|_{{\mathcal H}^1_\phi(\N\times\R)}$ and space
${\mathcal H}^1_\phi(\N\times\R)$ are defined similarly. These norms
are designed to accommodate the Fourier transforms of
the appearing singular measures, such as PLS.  Moreover, we
shall need the following norms on frequency marginals
\[
\|\widehat{g}\|_{L^1_\phi(\R^+)}={\displaystyle\int_{\R^+}} \phi(\tau)|\widehat{g}(\tau)|d\tau\quad\text{and}\quad
\|\widehat{g}\|_{{\mathcal H}^1_\phi(\R^+)}=\left({\displaystyle\int_{\R^+}} \phi(\tau)^2\left(|\widehat{g}(\tau)|^2+|\widehat{g}'(\tau)|^2\right)d\tau\right)^{\frac12}.
\]
Of note, together with $\hat{g}(-\tau)=\overline{\hat{g}(\tau)}$, the condition $\|\widehat{g}\|_{{\mathcal H}^1_{e^{a\tau}}(\R^+)}<+\infty$ implies, via the Paley-Wiener theorem, that $g$ must be analytic in a strip around the horizontal axis in $\C$.

While focus is on asymptotic stability of certain solutions, the developed exponential stability analysis also fits the setting of the theory of center manifolds in infinite dimension \cite{VI92}. In particular, for Banach spaces defined using weighted $L^\infty$-norms for Fourier transforms, a center-unstable manifold has been proved to exist for $K\sim K_c$, which attracts all trajectories of \eqref{KPDE} in a sufficiently small neighborhood of $f_\text{hom}$ (see Theorem 7 in \cite{D16a} and also \cite{C15}).

\subsection{Asymptotic relaxation to the homogeneous state}\label{S-HOMSTAB}
A unique homogeneous stationary state $f_\text{hom}(d\theta,d\omega)=\frac{g(\omega)}{2\pi}d\theta d\omega$ exists for every $g$ and $K$, and its order parameter vanishes $r_\text{hom}=0$. In addition to regularity requirements on $g$, the stability of this state relies on the following condition \cite{D16a,FG-VG16}
\begin{equation}
\frac{K}2\int_{\R^+}\hat{g}(\tau)e^{-z\tau}d\tau\neq 1,\ \forall z\in\C\ :\ \text{Re}(z)\geq 0
\label{STABHOMOG}
\end{equation}
which involves the Laplace transform of $\hat{g}$. When $g$ is symmetric and unimodal, this requirement is equivalent to the inequality $K<\frac{2}{\pi g(0)}$ mentioned in the Introduction. In the general case, the argument principle and the Plemelj formula can be employed to show that \eqref{STABHOMOG} holds under the following analogue of the Penrose criterion in the Vlasov literature:
\[
\ \text{\sl We have}\ K<\frac{2}{\pi g(\Omega)}\ \text{\sl for every}\ \Omega\in\R\ \text{\sl s.t.}\ \int_{\R^+}\frac{g(\Omega-\omega)-g(\Omega+\omega)}{\omega}d\omega=0.
\]
This criterion is however not necessary for stability; the tri-Cauchy distribution at the end of Section \ref{S-STABCOND} provides a counter-example.
\begin{Thm}
  Assume that $g\in C^2(\R)$ is such that $\|\widehat{g}\|_{L^1_{(1+\tau)^b}(\R^+)}<+\infty$ for some $b>1$ and let $K$ be so that \eqref{STABHOMOG} holds. Then, there exists $\epsilon>0$ such that for every $f(0)\in C^2(\T^1\times\R)$ with marginal density $g$ and satisfying $\|\widehat{f(0)}\|_{{\mathcal H}^1_{(1+\tau)^b}(\N\times\R^+)}<\epsilon$, we have, in the weak sense,
\[
\lim_{t\to +\infty}f(t)=f_\text{\rm hom}.
\]
\label{THMSTHOM}
\end{Thm}
\vspace*{-0.7cm}
This statement, which by definition of weak convergence, implies $\lim\limits_{t\to +\infty}r(t)=0$ for the order parameter associated with $f(t)$, is an immediate consequence of the combination of Theorems 5 and 39 in \cite{D16a}.

The stability condition \eqref{STABHOMOG} is optimal, as least as far as linear stability is concerned. This means that, if there exists $z_0\in\C$ with $\text{Re}(z_0)>0$ such that $\frac{K}2\int_{\R^+}\hat{g}(\tau)e^{-z_0\tau}d\tau=1$, then the linearized Kuramoto equation around $f_\text{hom}$ has a solution with exponentially growing order parameter \cite{SM91}.

The constraint $\|\widehat{f(0)}\|_{{\mathcal H}^1_{(1+\tau)^b}(\N\times\R^+)}<\epsilon$ actually impacts the perturbation $f(0)-f_\text{\rm hom}$ and is justified by the possible existence of PLS while \eqref{STABHOMOG} holds, see e.g.\ \cite{DFG-V16,MBSOSA09,OW13}. However, such coexistence can only happen for relatively strong interaction and the conclusion of Theorem \ref{THMSTHOM} can be asserted for any $C^4$ initial perturbation of finite weighted Sobolev norm ${\mathcal H}^4$, provided that $K$ is small enough (Proposition 3.2 in \cite{FG-VG16} combined with Theorems 39 in \cite{D16a}). Moreover, when focus is made on the observable $r(t)$, the uniform constraint
\[
K\leq \frac2{\|\hat{g}\|_{L^1(\R^+)}}\ \left( =  K_c\ \text{if}\ g\ \text{unimodal and symmetric}\right)
\]
ensures that $\lim\limits_{t\to +\infty}r(t)=0$ holds for the trajectory of every initial perturbation in ${\mathcal H}^1_{(1+\tau)^b}(\N\times\R^+)$ (see Section 4.4 in \cite{D16b}).

Asymptotic convergence of $f(t)$ in Theorem \ref{THMSTHOM}  is proved using accurate control of the relaxation rate of its Fourier transform.
As anticipated in \cite{SMM92}, the order parameter relaxation rate can be estimated based on the initial perturbation regularity.
\begin{Pro}
  (i) Under the conditions of Theorem \ref{THMSTHOM}, we have
  \[
    r(t)=O(t^{-b}).
  \]
  (ii) If, in addition,
  $\|\widehat{g}\|_{L^1_{e^{a\tau}}(\R^+)}<+\infty$ for some $a>0$,
  then, there exist $\epsilon,a'>0$ such that, for every
  $f(0)\in C^2(\T^1\times\R)$ with
  $\|\widehat{f(0)}\|_{{\mathcal
      H}^1_{e^{a\tau}}(\N\times\R)}<\epsilon$, we have
  \[
    r(t)=O(e^{-a' t}).
  \]
  \label{LANDAUDAMP}
\end{Pro}
\vspace*{-0.7cm}
Statement (i) (resp.\ (ii)) is a consequence of Theorem 5 (resp.\ 4) in {\rm \cite{D16a}}. Under the conditions of (ii), the essential spectrum of the linearized dynamics operator is contained in the half-plane $\text{Re}(z)\leq -a$. In the complement half-space, the spectrum consists of finitely many eigenvalues, all of them have negative real part. The rate $a'$ corresponds to these eigenvalue largest real part.

Finally, one can mention that the conclusions of (i) and (ii) apply under weaker assumptions on $g$ and $f(0)$ \cite{D16a}. These conditions express as pointwise constraints on $\hat{g}$ and $\widehat{f(0)}$ and imply similar pointwise decay estimates for $\widehat{f(t)}$. Besides, inspired by a similar analysis for the Vlasov-HMF equation \cite{FR16}, ref.\ \cite{FG-VG16} considers perturbations in the original space of measure densities via the weighted Sobolev norm defined by
\[
\sum_{k_\theta,k_\omega\geq 0,\ k_\theta+k_\omega\leq n}\|\langle \omega\rangle\partial_\theta^{k_\theta}\partial_\omega^{k_\omega}v\|_{L^2(\T^1\times\R)}^2, \quad \text{where}\quad \langle \omega\rangle=\sqrt{1+\omega^2}\ \text{and}\ n\geq 4.
\]
In this setting, polynomial Landau damping is obtained as in (i) above, as well as algebraic convergence in a twisted reference frame, of the Sobolev norm of the density of $f(t)$.

\subsection{Asymptotic relaxation to PLS}\label{S-ASYMPLS}
Recall  that $R_\Theta$ denotes the representation on measures of the rigid rotation on the cylinder. In full generality, PLS can be defined as solutions of \eqref{KPDE} of the form $t\mapsto R_{\Omega t}f_\text{pls}$, for some global frequency $\Omega\in\R$ and reference measure $f_\text{pls}$ with non-vanishing order parameter
\[
r_\text{pls}:=\int_{\T^1\times\R}e^{i\theta}f_\text{pls}(d\theta,d\omega)\neq 0.
\]
Galilean invariance implies that PLS are indifferent to the action of $R_\theta$; hence we may assume that the state $f_\text{pls}$ has real and positive order parameter $r_\text{pls}$ and consider the circle $\{R_\Theta f_\text{pls}\}_{\Theta\in\T^1}$. For the same reason, we may assume that $\Omega=0$, up to a translation of $g$. In this case, every $R_\Theta f_\text{pls}$ is a stationary state and one can show that the corresponding expression of $f_\text{pls}$ is \cite{ABP-VRS05,MS07,S00}
\[
f_\text{pls}(\theta,\omega)=\left\{\begin{array}{ccl}
\left(\alpha(\omega)\delta_{\arcsin(\frac{\omega}{Kr_\text{pls}})}(\theta)+(1{-}\alpha(\omega))\delta_{\pi-\arcsin(\frac{\omega}{Kr_\text{pls}})}(\theta)\right)g(\omega)&\text{if}&|\omega|\leq Kr_\text{pls}\\
\frac{\sqrt{\omega^2 - (Kr_\text{pls})^2}}{2\pi |\omega - Kr_\text{pls}\sin \theta|} g(\omega) &\text{if}&|\omega|> Kr_\text{pls}
\end{array}\right.
\]
which confirms the singular nature of PLS \cite{S00}.
Here the measurable function $\alpha:[-Kr_\text{pls},Kr_\text{pls}]\to [0,1]$ quantifies the relative contribution of the two equilibria $\arcsin(\frac{\omega}{Kr_\text{pls}})$ and $\pi-\arcsin(\frac{\omega}{Kr_\text{pls}})$ of the equation of characteristics
\begin{equation}
\dot\theta=\omega-Kr_\text{pls}\sin\theta.
\label{CHARACEQU}
\end{equation}
The PLS with $\alpha=1$ a.e.\ is denoted by $f_\text{s}$, and its order parameter by $r_\text{s}$. The equilibrium $\arcsin(\frac{\omega}{Kr_\text{pls}})$ is stable for the one-dimensional dynamics \eqref{CHARACEQU}, while the other one is unstable. This suggests that only $f_\text{s}$ can be stable among possible $f_\text{pls}$ \cite{MS07}. This argument can be formally justified using the norm above. Indeed, provided that $\|\widehat{g}\|_{{\mathcal H}^1_{e^{a\tau}}(\R)}<+\infty$, the only $f_\text{pls}$ whose Fourier transform lies in ${\mathcal H}^1_{e^{a\tau}}(\N\times\R)$ turns out to be $f_\text{s}$ (Proposition A.2 in \cite{DFG-V16}).

The explicit expression of $f_\text{pls}$ yields an existence condition, which materializes as a self-consistency condition on $r_\text{pls}$ \cite{DFG-V16,MS07,S00}. For $f_s$, this condition writes
\begin{equation}
\int_\R\beta\left(\frac{\omega}{Kr_\text{s}}\right)g(\omega)d\omega=r_\text{s}\quad\text{where}\quad
\beta(\omega)=-i\omega+\left\{\begin{array}{ccl}
\sqrt{1-\omega^2}&\text{if}&|\omega|\leq 1\\
i\omega\sqrt{1-\omega^{-2}}&\text{if}&|\omega|> 1.
\end{array}\right.
\label{EXISTPLS}
\end{equation}
Of note, if $g$ is symmetric around 0, then the imaginary part of the LHS here automatically vanishes and the existence condition becomes \cite{MS07,OW12,OW13,S00}
\[
\int_{-Kr_\text{s}}^{Kr_\text{s}}\beta\left(\frac{\omega}{Kr_\text{s}}\right)g(\omega)d\omega=r_\text{s}.
\]
If also $g\in C^0$, an analysis of this condition shows that a PLS
$f_\text{s}$ exists for every $K>K_c$
\cite{S00}. Moreover, if $g$ is unimodal, $f_\text{s}$ is unique and
does not exist for $K\leq K_c$. For more general frequency marginals, PLS
existence is not so simple but can be granted, once allowing for
$\Omega\neq 0$,  under the condition that the homogeneous
state is unstable, see Section \ref{S-STABCOND} below.

For the stability condition, the following notations are needed: given
$z\in\C$ with $\text{Re}(z)\geq 0$ and $r\in\R^+$, let $M(z,r)$ be the
$2\times 2$ matrix defined by
\[
M(z,r)=\left(\begin{array}{cc}
J_0(z,r)&J_2(z,r)\\
\overline{J_2(\bar{z},r)}&\overline{J_0(\bar{z},r)}\end{array}\right)\quad \text{with}\quad J_k(z,r)=\int_\R\frac{\beta^k\left(\frac{\omega}{Kr}\right)}{z+ i\omega+Kr\beta\left(\frac{\omega}{Kr}\right)}g(\omega)d\omega,
\]
(where the quantities $J_k$ are defined by continuity for $\text{Re}(z)=0$).
\begin{Thm}
Given $b>\tfrac32$, assume that $\|\hat{g}\|_{{\mathcal H}^1_{(1+\tau)^{b_g}}(\R)}<+\infty$ for some $b_g>b+3$ and let $K$ be such that a stationary PLS $f_\text{\rm s}$ with marginal density $g$ and order parameter $r_\text{\rm s}\in\R^+$ exists and satisfies
\begin{equation}
\left\{
\begin{aligned}
& \det \left(\text{\rm Id}-\frac{K}2 M(z,r_\text{\rm s})\right)\neq 0,\ \forall z\neq 0\ \text{with}\ \text{\rm Re}(z)\geq 0, \\
&  z=0\ \text{\rm is a simple zero of the function}\ z\mapsto \det \left(\text{\rm Id}-\frac{K}2 M(z,r_\text{\rm s})\right).
\end{aligned}
\right.
\label{STABCOND}
\end{equation}
Then, there exists $\epsilon>0$ such that for every $f(0)$ with marginal density $g$ and so that
\[
\|\widehat{f(0)}-\widehat{R_{\Theta}f_\text{\rm s}}\|_{{\mathcal H}^1_{(1+\tau)^{b}}(\N\times\R^+)}<\epsilon\ \text{for some}\ \Theta\in\T^1
\]
there exists $\Theta_\infty\in\T^1$ such that we have
\[
\lim_{t\to+\infty}f(t)=R_{\Theta_\infty} f_\text{s}\ \text{(weak sense)}\quad \text{and}\quad |r(t)-r_\text{s}e^{i\Theta_\infty}|=O(t^{\tfrac12-b}).
\]
\label{SOBOLEVPLS}
\end{Thm}
\vspace*{-0.7cm}
As for \eqref{STABHOMOG}, the stability condition \eqref{STABCOND} can
be shown to be optimal.  Moreover, the statement above is a
simplification of Theorem 2 in \cite{D17}, which includes broader
regularity conditions and provides quantitative control of the
convergence in Fourier space. As for $f_\text{hom}$, an analogous
statement holds in the exponential setting.
\begin{Thm}
  Assume that $\|\hat{g}\|_{{\mathcal H}^1_{e^{a\tau}}(\R^+)}<+\infty$
  for some $a>0$ and let $K$ be such that a stationary PLS
  $f_\text{\rm s}$ with marginal density $g$ and order parameter
  $r_\text{\rm s}\in\R^+$ exists and satisfies \eqref{STABCOND}. Then,
  there exist $\epsilon,a'>0$ such that for every $f(0)$ with marginal
  density $g$ so that
  \[
    \|\widehat{f(0)}-\widehat{R_{\Theta}f_\text{\rm s}}\|_{{\mathcal H}^1_{e^{a\tau}}(\N\times\R)}<\epsilon\ \text{for some}\ \Theta\in\T^1
  \]
  there exists $\Theta_\infty\in\T^1$ so that we have (in addition to weak convergence of measures)
  \[
    |r(t)-r_\text{s}e^{i\Theta_\infty}|=O(e^{-a't}).
  \]
  \label{ANALYTPLS}
\end{Thm}
\vspace*{-0.7cm}
This statement is a consequence of Theorem 2.1 in \cite{DFG-V16},
which claims the following convergence of Fourier transforms (and
hence the conclusion on the order parameter)
\[
\|\widehat{f(t)}-\widehat{R_{\Theta_\infty}f_\text{s}}\|_{{\mathcal H}^1_{e^{a\tau}}(\N\times\R)}=O(e^{-a't}).
\]
Detailed considerations on existence and stability of PLS will be
given in Section \ref{S-STABCOND}. When $g$ is unimodal and symmetric,
\eqref{STABCOND} turns out to coincide with the existence condition
$K>K_c$ \cite{DFG-V16}. In particular, Theorems \ref{SOBOLEVPLS} and
\ref{ANALYTPLS} complete the proof of the bifurcation diagram in Fig.\
\ref{BIFDIAG} left.

In addition, global stability can never hold for PLS because
$f_\text{hom}$ is a distinct stationary state, which satisfies
$\|\widehat{f_\text{hom}}\|_{{\mathcal
    H}^1_{(1+\tau)^{b}}(\N\times\R^+)}<+\infty$ under the conditions
of Theorem \ref{SOBOLEVPLS} (or \ref{ANALYTPLS}).  Notice finally that
for $r_\text{s}\to 0$, not only the PLS expression reduces to that of
$f_\text{hom}$, but PLS existence and stability conditions converge as
well. We kept the exposition of stationary states separated for
historical and pedagogical reasons.

\section{Main ingredients of proofs}\label{S-PROOF}
The asymptotic stability of stationary states has been proved using
the formulation of the dynamics in Fourier space. Instead of providing
all details, we focus here on the decay of the order parameter under
the linearized dynamics, which turns out to be governed by a Volterra
equation of the second kind. Conditions \eqref{STABHOMOG} and
\eqref{STABCOND}, and asymptotic decay as given in Proposition
\ref{LANDAUDAMP} and in Theorems \ref{SOBOLEVPLS} and \ref{ANALYTPLS},
then follow from the corresponding theory \cite{GLS90}.  In addition,
we will comment on how to deal with the nonlinear terms in the
exponential case.

\subsection{Volterra equation for the order parameter}
Let $u=\{u_\ell\}_{\N}=\{u_\ell(\tau)\}_{\N\times \R}$ be an initial
perturbation with $u_0(\tau)=0$ (so that the frequency marginal is
preserved). Inserting the expression $\widehat{f_\text{s}}+u$ in the
Kuramoto dynamics in Fourier space yields the following evolutionary
equation
\begin{equation} \label{KF_condensed}
\partial_t u=L_1u+L_2u+Qu,
\end{equation}
\[
\text{where}\quad (L_1u)_\ell=\ell\left(\partial_\tau u_\ell+\frac{K r_\text{s}}2\left(u_{\ell-1}-u_{\ell+1}\right)\right)
\quad
\text{and}
\quad
(L_2u)_\ell=\frac{K \ell}2\left(u_1(0)(\widehat{f_\text{s}})_{\ell-1}-\overline{u_1(0)}(\widehat{f_\text{s}})_{\ell+1}\right),
\]
and the operator $Q$ collects the nonlinear terms
\[
(Qu)_\ell=\frac{K \ell}2\left(u_1(0)u_{\ell-1}-\overline{u_1(0)}u_{\ell+1}\right).
\]
Naturally, for $r_\text{s}=0$ and
$\widehat{f_\text{s}}=\widehat{f_\text{hom}}$, these equations
describe the perturbation dynamics around $f_\text{hom}$. In this
case, we have the simplification
$(L_1u)_\ell=\ell\partial_\tau u_\ell$ and
$(L_2u)_\ell=\frac{K}2u_1(0)\widehat{g}\delta_{\ell,1}$, while $Q$
remains unchanged.

Prior to any other consideration, this perturbation dynamics needs to
be granted well-posed in the weighted norm setting. In this respect,
Proposition 3.1 in \cite{DFG-V16} claims that the subset of measures
with Fourier transforms in ${\mathcal H}^1_{e^{a\tau}}(\N\times\R)$
has well-defined Kuramoto dynamics and is invariant under the flow
(see \cite{D17} for a similar well-posedness result in the algebraic
setting).

Moreover, one needs to incorporate the fact that
$L_2$ is only $\R$-linear and not $\C$-linear if
$\widehat{f_\text{s}}\neq \widehat{f_\text{hom}}$. One way to proceed
is to treat the real and imaginary components separately
\cite{D17,MS07,OW13}. Here, we adopt a different but equivalent
approach that substitutes complex conjugates by an independent
variable. Given $u=\{u_\ell(\tau)\}_{\N\times\R}$ and
$v=\{v_\ell(\tau)\}_{\N\times\R}$ (which is a substitute for
$\bar{u}$), let
\[
\text{\sl u}=\{\text{\sl u}_\ell(\tau)\}_{\N\times\R}\quad \text{where}\quad \text{\sl u}_\ell(\tau)={u_\ell(\tau) \choose v_\ell(\tau)}\in\C^2,\ \forall (\ell,\tau)\in\N\times\R
\]
and consider the $\C$-linear operators ${\cal L}_i$ ($i=1,2$) defined by
\[
({\cal L}_1\text{\sl u})_\ell(\tau)={(L_1u)_\ell(\tau)\choose (L_1v)_\ell(\tau)}
\quad\text{and}\quad
({\cal L}_2\text{\sl u})_\ell(\tau)=\frac{K}2
\left(\begin{matrix}(u_{\text{s},-})_\ell(\tau)&-(u_{\text{s},+})_\ell(\tau)\\[4pt]
-\overline{(u_{\text{s},+})_\ell(\tau)}&\overline{(u_{\text{s},-})_\ell(\tau)}\end{matrix}\right)\text{\sl u}_1(0)
\]
using the notations $(u_{\text{s},-})_\ell=\ell (\widehat{f_\text{s}})_{\ell-1}$ and $(u_{\text{s},+})_\ell=\ell (\widehat{f_\text{s}})_{\ell+1}$.
These extended operators are defined in such a way that when $v_\ell=\overline{u_\ell}$, we have
\[
({\cal L}_i\text{\sl u})_\ell(\tau)=\left(\begin{matrix}(L_iu)_\ell(\tau)\\[3pt]
\overline{(L_iu)_\ell(\tau)}\end{matrix}\right),\ \text{for}\ i=1,2.
\]
Moreover, it can be checked that this extension does not generate unstable spurious modes \cite{DFG-V16}.

Considerations on its resolvent in
${\mathcal H}^1_{e^{a\tau}}(\N\times\R)$ imply that $L_1$ generates a
$C^0$-semigroup in this space \cite{DFG-V16}. In case of
$r_\text{s}=0$, the semigroup is the free transport , namely
\[
(e^{tL_1}u)_\ell(\tau)=u_\ell(\tau+\ell t),\ \forall (t,\ell,\tau)\in \R^+\times\N\times\R.
\]
In ${\mathcal H}^1_{(1+\tau^b)}(\N\times\R)$, the semigroup admits
weak solutions and this is enough for our purpose. The same properties
directly follow for the semigroup of ${\cal L}_1$ in the corresponding
product spaces.

Both $u_{\text{s},-}$ and $u_{\text{s},+}$ belong to
${\mathcal H}^1_{e^{a\tau}}(\N\times\R)$ \cite[Proposition
A.2]{DFG-V16}; hence the operator ${\cal L}_2$ must be
bounded. Therefore, ${\cal L}_1+{\cal L}_2$ similarly generates a
$C^0$-semigroup. In the product space associated with
${\mathcal H}^1_{(1+\tau^b)}(\N\times\R)$, this semigroup is also
well-defined \cite{D17}.

When regarding $t\mapsto {\cal L}_2\text{\sl u}(t)$ as a forcing term in the linearized PDE
\[
\partial_t \text{\sl u}={\cal L}_1\text{\sl u}+{\cal L}_2\text{\sl u},
\]
Duhamel's principle implies that the solution writes
\[
\text{\sl u}(t)=e^{t{\cal L}_1}\text{\sl u}(0)+\int_0^te^{(t-s){\cal L}_1}{\cal L}_2\text{\sl u}(s)ds,\ \forall t\in\R^+.
\]
Using the linearity of $e^{tL_1}$ and the expression of ${\cal L}_2$,
a self-consistent equation results for the coordinate
$(\ell,\tau)=(1,0)$ of the solution's component
$\text{\sl u}(t)=\{\text{\sl u}_\ell(t,\tau)\}$. This equation is the
following Volterra equation
\begin{equation}
\text{\sl u}_1(t,0)-\left({\cal K}\ast \text{\sl u}_1(\cdot,0)\right)(t)=I(t),
\label{VOLTEQ}
\end{equation}
where the two-dimensional convolution by the kernel ${\cal K}$ is defined by
\[
\left({\cal K}\ast \text{\sl u}_1(\cdot,0)\right)(t)=\int_0^t{\cal K}(t-s)\text{\sl u}_1(s,0)ds
\quad
\text{with}
\quad
{\cal K}(t)=\frac{K }2
\left(\begin{matrix}(e^{tL_1}u_{\text{s},-})_1(0)&-(e^{tL_1}u_{\text{s},+})_1(0)\\[4pt]
-\overline{(e^{tL_1}u_{\text{s},+})_1(0)}&\overline{(e^{tL_1}u_{\text{s},-})_1(0)}\end{matrix}\right),
\]
and where $I(t)=(e^{t{\cal L}_1}\text{\sl u}(0))_1(0)$ is regarded as an input (which contains the initial perturbation). In particular, for inputs with conjugated initial components $\{v_\ell(0,\tau)\}=\{\overline{u_\ell(0,\tau)}\}$, we have $\text{\sl u}_1(t,0)={\displaystyle {\overline{r(t)}\choose r(t)}}$ and \eqref{VOLTEQ} describes the linearized evolution of the perturbation order parameter $r(t)$. In the case $r_\text{s}=0$, we have the simplification $(u_{\text{s},-})_\ell=\widehat{g}\delta_{\ell,0}$ and $(u_{\text{s},+})_\ell=0$, and the order parameter linearized trajectory is governed by the following one-dimensional Volterra equation
\begin{equation}
\overline{r(t)}-\frac{K}2(\widehat{g}\ast r)(t)=u_1(0,t)
\label{SOLVOLT1D}
\end{equation}
where $\ast$ now denotes the standard convolution of complex functions.

\subsection{Asymptotic decay of solutions and for stability conditions}
Volterra equations have unique and explicit solutions provided that
their kernel and forcing are locally bounded (Sect.\ 3, Chap.\ 2 in
\cite{GLS90}). In the case of \eqref{VOLTEQ}, these properties are
granted by the fact that $e^{tL_1}$ is itself locally bounded either
in ${\mathcal H}^1_{e^{a\tau}}(\N\times\R)$ or in
${\mathcal H}^1_{(1+\tau)^b}(\N\times\R^+)$, together with properties
of the states $u_{\text{s},-}$ and $u_{\text{s},+}$. The solution then
writes
\[
\text{\sl u}_1(t,0)=I(t) + ({\cal R}_{\cal K}\ast I)(t)\quad \text{where}\quad  {\cal R}_{\cal K}=\sum_{k=1}^{+\infty}{\cal K}^{\ast k}
\]
and the definition of the resolvent ${\cal R}_{\cal K}$ relies on the induction ${\cal K}^{\ast (k+1)}={\cal K}\ast{\cal K}^{\ast k}$ with ${\cal K}^{\ast 1}={\cal K}$.

\subsubsection{Analysis of the one-dimensional Volterra equation}
Let ${\cal R}_{\frac{K}2\widehat{g}}$ be the resolvent of the convolution by $\frac{K}2\widehat{g}$. The equation \eqref{SOLVOLT1D} has the solution
\[
\overline{r(t)}=u_1(0,t)+({\cal R}_{\frac{K}2\widehat{g}}\ast u_1(0,\cdot))(t)
\]
whose asymptotic properties are readily accessible, using the following weighted norm
\[
\|u\|_{L_{\phi}^\infty(\R^+)}=\text{ess\ sup}_{t\in\R^+}\phi(t) |u(t)|\quad \text{where}\quad \phi:\R^+\to\R^+.
\]
The weights $\phi(t)=e^{at}$ and $\phi(t)=(1+t)^b$ are
sub-multiplicative functions. Together with Young's inequality, this
property implies the following inequality \cite{D16a}
\[
\|r\|_{L_{\phi}^\infty(\R^+)}\leq \left(1+\|{\cal R}_{\frac{K}2\widehat{g}}\|_{L_{\phi}^1(\R^+)}\right)\|u_1(0,\cdot)\|_{L_{\phi}^\infty(\R^+)}.
\]
Therefore, if we can ensure that
$\|{\cal R}_{\frac{K}2\widehat{g}}\|_{L_{\phi}^1(\R^+)}<+\infty$, then
quantified order parameter decay
$\|r\|_{L_{\phi}^\infty(\R^+)}<+\infty$ will follow from a similar
feature $\|u_1(0,\cdot)\|_{L_{\phi}^\infty(\R^+)}<+\infty$ of the
initial perturbation (By Sobolev embedding, the latter holds provided
that $u(0)\in {\mathcal H}^1_\phi(\N\times\R^+)$). In particular, the
conclusions in Proposition \ref{LANDAUDAMP} for the linear dynamics
will be instances of this property, when applied to the exponential
and polynomial weights, respectively.

It remains to connect the constraint
$\|{\cal R}_{\frac{K}2\widehat{g}}\|_{L_{\phi}^1(\R^+)}<+\infty$ to
the condition \eqref{STABHOMOG} in each case. This equivalence is
given by the half-line Gelfand theorem (Theorem 4.3, Chapter 4 in
\cite{GLS90}). Indeed, since $\phi$ is sub-multiplicative and the
measure $\widehat{g}(t)dt$ is absolutely continuous, this statement
implies that the desired constraint holds under the conditions
$\widehat{g}\in L^1_{\phi}(\R^+)$ and
\begin{equation}
\frac{K}2\int_{\R^+}\hat{g}(t)e^{-zt}dt\neq 1,\ \forall z\in\C\ :\ \text{Re}(z)\geq -\lim_{t\to+\infty}\frac{\ln\phi(t)}{t}.
\label{HALFLINE}
\end{equation}

The conditions of Prop.\ \ref{LANDAUDAMP} {\sl (i)} then immediately follow. In the exponential case, the constraint \eqref{HALFLINE} appears more stringent than  \eqref{STABHOMOG} (because $-{\displaystyle\lim_{t\to+\infty}\frac{\ln\phi(t)}{t}}<0$).
To see that the conditions of Prop.\ \ref{LANDAUDAMP} {\sl (ii)} suffice, observe that $\|\widehat{g}\|_{L^1_{e^{a\tau}}(\R^+)}<+\infty$ implies that the Laplace transform
\[
z\mapsto \int_{\R^+}\hat{g}(t)e^{-zt}dt
\]
is holomorphic in every half-plane $\text{Re}(z)>-a$ and continuous up
to the boundary $\text{Re}(z)=-a$. By the Riemann-Lebesgue lemma, this
function must be uniformly small outside a sufficiently large
rectangular region of the form
\[
\text{Re}(z)\in [-a,A],\ |\text{Im}(z)|\leq B
\]
In particular, it cannot reach the value $\frac{2}{K}$ outside this
domain. By analyticity it can only reach this value at finitely many
points with $\text{Re}(z)>-a$. Assuming \eqref{STABHOMOG}, each of
these points must satisfy $\text{Re}(z)<0$. Therefore, all these
points must satisfy $\text{Re}(z)\leq -a'$ for some $a'\in (0,a)$. It
follows that \eqref{STABHOMOG} implies \eqref{HALFLINE} for
$\phi(t)=e^{a't}$, as desired.

\subsubsection{Analysis of the two-dimensional Volterra equation}
That PLS come in circles of stationary solutions implies that the
linearized dynamics at $\widehat{f_\text{s}}$ should be neutral with
respect to perturbations that are tangent to the circle
\cite{MS07}. In fact, we have \cite{DFG-V16}
\[
(L_1+L_2)u=0\ \text{for}\ u=\frac{d\widehat{R}_\Theta}{d\Theta}\big{|}_{\Theta=0}\widehat{f_\text{s}},\ \text{ie.}\ u_\ell=i\ell (\widehat{f_\text{s}})_\ell.
\]
As a consequence, the solution of \eqref{VOLTEQ} cannot be decaying when the initial input lies along the corresponding 0-eigenmode of ${\cal L}_1+{\cal L}_2$ in the product space. Asymptotic decay can only hold for solutions whose input is initially transversal to this direction. In order to integrate this constraint, one should require that the analogue condition to \eqref{HALFLINE} excludes the eigenvalue 0, ie.
\[
\det \left(\text{Id}-\int_{\R^+}{\cal K}(t)e^{-zt}dt\right)\neq 0,\ \forall z\neq 0\ :\ \text{Re}(z)\geq 0.
\]
Using that the Laplace transform of a semigroup is the resolvent of its generator and proceeding with algebraic manipulations on this resolvent \cite[Lemma 4.4]{DFG-V16}, yield the equality
\[
\int_{\R^+}{\cal K}(t)e^{-zt}dt=\frac{K}2M(z,r_\text{s})
\]
and the first constraint in \eqref{STABCOND} follows suit. Together
with imposing that 0 is a simple eigenvalue (second constraint in
\eqref{STABCOND}) and the conditions on $\widehat{g}$ in Theorem
\ref{SOBOLEVPLS} (resp.\ \ref{ANALYTPLS}), another result in the
theory of Volterra equations (Theorem 3.7, Chapter 7 in \cite{GLS90}
and subsequent comment) implies that the resolvent writes
\[
R_{\cal K}(t)=C+q(t),
\]
where $C$ is the constant $2\times2$ matrix corresponding to the 0-eigenmode above and where
\[
\|q\|_{L^1_\phi(\R^+)}<+\infty
\]
for $\phi(t)=(1+t)^b$ (resp.\ $\phi(t)=e^{at}$). The desired damping for transversal perturbations then results from the following inequality (obtained using a similar reasoning as above)
\[
\|r\|_{L_{\phi}^\infty(\R^+)}\leq \left(1+\|q\|_{L_{\phi}^1(\R^+)}\right)\|I(t)\|_{L_{\phi}^\infty(\R^+)}
\]
where $I(t)$ now denotes the first component of the input $(e^{t{\cal L}_1}\text{\sl u}(0))_1(0)$ when assuming conjugated components $\{v_\ell(0,\tau)\}=\{\overline{u_\ell(0,\tau)}\}$ in the initial perturbation $\text{\sl u}(0)$.

Notice finally that, unlike in the previous section, estimates on $\|I(t)\|_{L_{\phi}^\infty(\R^+)}$ are not immediate here, even when $\|u_1(0,\cdot)\|_{L_{\phi}^\infty(\R^+)}<+\infty$. In the exponential case, these estimates follow from the semigroup exponential stability in ${\mathcal H}^1_{e^{a\tau}}(\N\times\R)$ \cite{DFG-V16}, namely
\[
\|e^{tL_1}\|_{{\mathcal H}_{e^{a\tau}}^1(\N\times \R)}=O(e^{-a''t})
\]
for some $a''\in (0,a)$; hence the rate $a'<a''$ in Theorem \ref{ANALYTPLS}, when combined with \eqref{STABCOND} and similar analycity arguments to those in the previous section. In the algebraic case, no such property can exist for $e^{tL_1}$ in ${\mathcal H}^1_{(1+\tau)^b}(\N\times\R^+)$ (otherwise, we would have  exponential decay in this space).
Instead, an energy estimate yields the following inequality (see Lemma 4 in \cite{D17})
\[
  \|e^{tL_1}u\|^2_{{\mathcal H}^1_{(c+t+\tau)^b}(\N\times\R^+)}
  + \int_0^t|(e^{sL_1}u)_1(0)|^2(c+s)^{2b}ds
  \leq \|u\|^2_{{\mathcal H}^1_{(c+\tau)^b}(\N\times\R^+)},
  \ \forall c\geq 1,t\in\R^+
\]
and algebraic decay follows from a control of the integral using the Cauchy-Schwarz inequality.

\subsection{Control of nonlinear terms in the exponential case}
With full understanding at the linear level, nonlinearities remain to
be accounted for. Here, focus is made on PLS. Similar considerations
apply for $f_\text{hom}$. The fact that PLS come in circles requires
to get rid of the angular coordinate and to consider the radial
dynamics only \cite{HI11}. It can be shown that the Fourier transform
$\widehat{f}$ of any measure close enough to
$\{\widehat{R}_\Theta \widehat{f_\text{s}}\}_{\Theta\in\T^1}$ can be
written
\[
\widehat{f}=\widehat{R}_\Theta \left(\widehat{f_\text{s}}+u\right)\quad \text{where}\quad (\Theta,u)\in\T^1\times P_s({\mathcal H}^1_{e^{a\tau}}(\N\times \R))
\]
and $P_s$ is an appropriate projection on the complement of $\text{Ker}(L_1+L_2)$ \cite{DFG-V16}.
By inserting this expression in \eqref{KF_condensed} and by applying $P_s$, the following nonlinear equation results
\begin{equation}
\partial_t u=L_1u+L_2u+P_sQ'u
\label{PERTURB2}
\end{equation}
where $Q'$ is an updated nonlinear term, independent of the angular variable $\Theta$. In the case of $f_\text{hom}$, no projection is needed, and the considerations below apply {\sl mutatis mutandis} to \eqref{KF_condensed}.

The previous section showed that \eqref{STABCOND} implies in the exponential case that the semigroup $e^{t(L_1+L_2)}$ is exponentially stable, namely
\[
\|e^{t(L_1+L_2)}\|_{P_s({\mathcal H}^1_{e^{a\tau}}(\N\times \R))}=O(e^{-a't}).
\]
In a standard proof of sink asymptotic stability, the nonlinear terms
are assumed to be sufficiently regular, say $C^2$, so that for
sufficiently small perturbations, they can be dominated by the linear
exponential stability and exponential decay of the full system
solution follows.

Unfortunately, the nonlinearity $Q$ (and hence $Q'$) is not regular at all in ${\mathcal H}^1_{e^{a\tau}}(\N\times \R)$; in fact it does not even map this space into itself. Instead, we have
\[
Q:{\mathcal H}^1_{e^{a\tau},0}(\N\times \R)\mapsto {\mathcal H}^1_{e^{a\tau},1}(\N\times \R)\quad\text{where}\quad
{\mathcal H}^1_{\phi,k}(\N\times\R)=\{u\in \C^{\N\times\R}\ :\ \|u\|_{{\mathcal H}^1_{\phi,k}(\N\times\R)}<+\infty\}
\]
and the new norm is an extension of the one introduced in Section \ref{S-RESULTS}
\begin{equation}
\|u\|_{{\mathcal H}^1_{\phi,k}(\N\times\R)}=\left(\sum_{\ell\in\N}\int_\R \ell^{2k}\phi(\tau)^2\left(|u_\ell(\tau)|^2+|u'_\ell(\tau)|^2\right)d\tau\right)^{\frac12}.
\label{EXTRAN}
\end{equation}
Nonetheless, the linear terms \eqref{PERTURB2} have enough
regularizing effect to dominate nonlinearities. In fact, when
regarding the nonlinearity in \eqref{PERTURB2} as a forcing term, the
following adaptation of the Gearhart-Pr\"uss Theorem shows asymptotic
decay.  Given an Hilbert space $H$ with norm $\|\cdot\|_H$, a number
$\gamma\in\R^+$, and a mapping $w:\R\to H$, consider the norm defined
by
\[
\|w\|_{H,\gamma}=\left(\int_{\R^+}e^{2\gamma t}\|w(t)\|_{H}^2dt\right)^{\frac12}.
\]
\begin{Lem} {\rm \cite{DFG-V16}} Let $X\hookrightarrow Y$ be Hilbert
  spaces and $A$ be a densely defined linear operator that generates a
  $C^0$-semigroup on both $X$ and $Y$. Assume the existence of
  $\gamma\in\R^+$ such that the resolvent of $A$ over both spaces
  contains the half-plane $\text{Re}(\lambda) \geq -\gamma$ and
  satisfies
  \begin{equation*}
    \sup_{y \in \R}  \| ((-\gamma + i y) \text{\rm Id} - A)^{-1}
    \|_{Y \to X} <+\infty.
  \end{equation*}
  Then the unique mild solution $w \in C(\R^+,Y)$ of the initial value problem
  \[
    \frac{dw}{dt}=Aw+G
  \]
  assuming $\|G\|_{Y,\gamma}<+\infty$ and $\|w_\text{\rm in}\|_X<+\infty$ for $w(0)=w_\text{\rm in}$, has the following properties
  \begin{itemize}
  \item[$\bullet$] $w(t) \in X$ for a.e.\ $t \in \R^+$
  \item[$\bullet$] $\|w\|_{X,\gamma}\leq C\left(\|w_\text{in}\|_X +\|G\|_{Y,\gamma}\right)$ for some $C\in\R^+$.
  \end{itemize}
\end{Lem}
In short terms, under a suitable property of the resolvent, asymptotic
decay of the solution can be ensured in $X$, even though control of
the forcing only holds in the larger space $Y$.  Now, one checks that
the forced linear equation associated with \eqref{PERTURB2} satisfies
the conditions of this Lemma, in appropriate product spaces, with
$A={\cal L}_1+{\cal L}_2$ \cite{DFG-V16}. Therefore, its solution
satisfies
\[
\|u\|_{{\mathcal H}^1_{e^{a\tau}}(\N\times \R),a'}<+\infty.
\]
To conclude the proof, it remains to show, using a localization
procedure, that this $L^2$-control in time implies $L^\infty$-control
for the original nonlinear equation, see section 5.4 in
\cite{DFG-V16}.

\section{Convergence to the Ott-Antonsen manifold}\label{S-OA}
In addition to the basic characteristics listed in the Introduction, \eqref{KPDE} has another remarkable feature. Its solutions asymptotically approach the so-called Ott-Antonsen (OA) manifold, namely the set of measures for which the Fourier transform $\{\widehat{f}_\ell(\tau)\}_{(\ell,\tau)\in \N\times\R}$ satisfies \cite{OA08},
\[
\widehat{f}_\ell=h^{\ast \ell}\ast \widehat{g},\ \forall \ell\in\N\cup\{0\},
\]
where $g$ is the frequency marginal associated with $f$, $h:\R\to\C$ is arbitrary, and $\ast$ now denotes the convolution on the whole line, ie.
\[
(u\ast v)(\tau) = \int_{\R} u(\tau-\sigma)v(\sigma)d\sigma,\ \forall  \tau\in\R.
\]
and $h^{\ast (\ell+1)}=h\ast h^{\ast \ell}$ for $\ell\in\Z^+$, where $h^{\ast 0}$ is the Dirac distribution.

Several motivations for the OA manifold have been given in the
literature. As mentioned in the Introduction, the dynamics in this set
is a finite-dimensional system when $g$ is meromorphic with finitely
many poles in the lower half-plane \cite{MBSOSA09,OA08} (see also
Section 5.6.2 in \cite{D16b}). Moreover, this set captures the order
parameter dynamics \cite{OA09,OHA11}. In addition, it selects suitable
candidates for PLS stability, namely $f_\text{s}$ is the only PLS
$f_\text{pls}$ contained in this set. Finally, when evaluated in the
OA manifold, the corresponding PLS stability condition results to be
identical to \eqref{STABCOND} \cite{DFG-V16,OW13}.  Here, we provide a
full proof that the OA manifold is a global attractor for appropriate
measures. In order to formulate the statement, we first observe that
this set can be regarded as the set of measures for which all
functions
\[
w_{n,m}=\widehat{f}_{n+m}\ast\widehat{g}-\widehat{f}_n\ast\widehat{f}_m,\ \forall n,m\in\N\cup\{0\}
\]
identically vanish. Accordingly the distance to this set will be evaluated using
\[
\|w\|_{{\mathcal H}^1_{e^{a\tau}}(\N^2\times\R)}=\left(\sum_{n,m\in\N}\int_\R \frac{e^{2a\tau}}{nm}\left(|w_{n,m}(\tau)|^2+|w'_{n,m}(\tau)|^2\right)d\tau\right)^{\frac12}.
\]
\begin{Pro}
Assume that $f(0)$ is such that $\|w(0)\|_{{\mathcal H}^1_{e^{a\tau}}(\N^2\times\R)}<+\infty$ for some $a>0$ and that the corresponding global solution $t\mapsto f(t)$ exists. Then $\|w(t)\|_{{\mathcal H}^1_{e^{a\tau}}(\N^2\times\R)}<+\infty$ for all $t\in\R^+$ and
\[
\|w(t)\|_{{\mathcal H}^1_{e^{a\tau}}(\N^2\times\R)}\leq \|w(0)\|_{{\mathcal H}^1_{e^{a\tau}}(\N^2\times\R)}\ e^{-at} .
\]
In particular, the following limit holds
\[
\lim_{t\to +\infty}w_{n,m}(t,\tau)=0,\ \forall n,n\in\N, \tau\in\R.
\]
\label{ATTRACTOA}
\end{Pro}
\vspace*{-0.7cm}
\begin{proof}
Using the relations
\[
\partial_\tau w_{n,m}=\partial_\tau \widehat{f}_{n+m}\ast\widehat{g}-\partial_\tau \widehat{f}_n\ast\widehat{f}_m=\partial_\tau \widehat{f}_{n+m}\ast\widehat{g}-\widehat{f}_n\ast\partial_\tau \widehat{f}_m
\]
and the Kuramoto dynamics in Fourier space, the following evolutionary equation results
\[
\partial_t w_{n,m}=(n+m)\partial_\tau w_{n,m}+\frac{K}2\left(\overline{r(t)}\left(n w_{n-1,m}+mw_{n,m-1}\right)-r(t)\left(n w_{n+1,m}+m w_{n,m+1}\right)\right).
\]
Together with the relations $w_{0,m}\equiv w_{n,0}\equiv 0$, this equation yields after standard manipulations
\begin{align*}
\frac{d}{dt}\sum_{n,m\in\N}\int_\R \frac{e^{2a\tau}}{nm}|w_{n,m}(\tau)|^2d\tau&=2\sum_{n,m\in\N}\int_\R \frac{(n+m)e^{2a\tau}}{nm}\text{Re}\left(w_{n,m}(\tau)\overline{w'_{m,n}}(\tau)\right)d\tau\\
&=-2a\sum_{n,m\in\N}\int_\R \frac{(n+m)e^{2a\tau}}{nm}|w_{n,m}(\tau)|^2d\tau.
\end{align*}
The same computations hold with $w'_{n,m}$ instead of $w_{n,m}$. Adding the two results, it follows that
\[
\frac{d}{dt}\|w\|^2_{{\mathcal H}^1_{e^{a\tau}}(\N^2\times\R)}\leq -2a \|w\|^2_{{\mathcal H}^1_{e^{a\tau}}(\N^2\times\R)}
\]
from where the first part of the Proposition results. The second part is a direct consequence of the first result together with the Sobolev embedding ${\mathcal H}^1([-\tau,\tau])\hookrightarrow C_0([-\tau,\tau])$, for every $\tau\in\R^+$.
\end{proof}

We do not know if the conditions of the Proposition hold for every $f(0)$ such that $\widehat{f(0)}\in {\mathcal H}^1_{e^{a\tau}}(\N\times\R)$. Yet, the next statement provides sufficient conditions for the Proposition to apply.
\begin{Lem}
Suppose that $\|\hat{g}\|_{{\mathcal H}^1_{e^{a'\tau}}(\R^+)}<+\infty$ and $\|\widehat{f(0)}\|_{{\mathcal H}^1_{e^{a'\tau}}(\N\times\R)}<+\infty$ for all $a'\in [a-\epsilon,a+\epsilon]$ where $\epsilon>0$ is (arbitrarily) small. Then, $f(t)$ satisfies the assumptions of Proposition \ref{ATTRACTOA} for every $t>0$.
\end{Lem}
 \begin{proof}
By splitting the convolution integral into the sum of an integral over $\R^-$ and one over $\R^+$, one easily gets the following estimate given any two functions $u,v:\R\to\C$
\[
\|u\ast v\|^2_{L^2_{e^{a\tau}}(\R)}\leq \frac{1}{2(a-a_-)}\|u\|^2_{L^2_{e^{a_-\tau}}(\R)}\|v\|^2_{L^2_{e^{a_-\tau}}(\R)}+\frac{1}{2(a_+-a)}\|u\|^2_{L^2_{e^{a_+\tau}}(\R)}\|v\|^2_{L^2_{e^{a_+\tau}}(\R)}
\]
for every $a_-<a<a_+$. Using this inequality in straightforward computations based on the definition of $w_{n,m}$, together with the notation in \eqref{EXTRAN} and the estimate
\[
\sum_{n,m\in\N}\frac{u^2_{n+m}}{nm}=\sum_{\ell=2}^{+\infty}u^2_\ell\sum_{n=1}^{\ell-1}\frac1{n(\ell-n)}\leq 2\sum_{\ell=2}^{+\infty}\frac{u^2_\ell(1+\log \ell)}{\ell}\leq 2\sum_{\ell=2}^{+\infty}u^2_\ell.
\]
we obtain
\begin{align*}
\|w(t)\|^2_{{\mathcal H}^1_{e^{a\tau}}(\N^2\times\R)}&\leq \frac{2\|\widehat{g}\|^2_{L^2_{e^{a_-\tau}}(\R)}+\|\widehat{f(t)}\|^2_{L^2_{e^{a_-\tau}}(\N\times\R)}}{a-a_-}\|\widehat{f(t)}\|^2_{{\mathcal H}^1_{e^{a_-\tau}}(\N\times\R)}\\
&+\frac{2\|\widehat{g}\|^2_{L^2_{e^{a_+\tau}}(\R)}+\|\widehat{f(t)}\|^2_{L^2_{e^{a_+\tau}}(\N\times\R)}}{a_+-a}\|\widehat{f(t)}\|^2_{{\mathcal H}^1_{e^{a_+\tau}}(\N\times\R)}
\end{align*}
The assumptions of the Lemma and the fact that the Cauchy problem is well-posed in ${\mathcal H}^1_{e^{a'\tau}}(\N\times\R)$ for $a'\in \{{a_-,a_+\}}\subset [a-\epsilon,a+\epsilon]$ \cite[Proposition 3.1]{DFG-V16} imply that the second terms in the RHS above are bounded for every $t>0$.
\end{proof}

\section{Existence, stability and bifurcations}\label{S-STABCOND}
This section investigates the connections between the various existence and stability conditions of Section \ref{S-RESULTS} and discusses their concrete materialization in some examples.

For $g$ symmetric and unimodal, instability of $f_\text{hom}$ is
equivalent to existence and stability of stationary PLS (Fig.\
\ref{BIFDIAG} left). In other cases, this connection is not so
tight. In particular, stable PLS could exist while $f_\text{hom}$ is
stable. This happens for instance for the bi-Cauchy distribution
\[
g_{\Delta,\Omega}(\omega)=\frac{\Delta}{2\pi}\left(\frac1{(\omega-\Omega)^2+\Delta^2}+\frac1{(\omega+\Omega)^2+\Delta^2}\right),
\]
when $\Omega>\frac{\Delta}{\sqrt{3}}$ ($\Delta\in\R^+$) so that the
distribution is bimodal. In this case, $f_\text{hom}$ is stable for
all $K\leq K_c$. Yet, stable and unstable stationary PLS $f_\text{s}$
co-appear at some $K<K_c$. Moreover, the unstable PLS branch merges
with $f_\text{hom}$ via sub-critical bifurcation at $K_c$ (Fig.\
\ref{BIFDIAG}, right and see \cite{DFG-V16,MBSOSA09} for details).
While this example shows that $f_\text{hom}$ instability is not
necessary for PLS existence, the next statement shows that it is
sufficient, provided that globally rotating PLS are allowed.
\begin{Pro}
  Assume that $g$ is Lipschitz continuous and such that
  $\|\widehat{g}\|_{L^1_{1+\tau}(\R^+)}<+\infty$ and that
  \eqref{STABHOMOG} fails for some $z$ with $\text{Re}(z)>0$. Then, a
  PLS exists for some frequency $\Omega\in\R$ and profile of type
  $f_\text{\rm s}$.
\end{Pro}
We can have $\Omega\neq 0$ even though $g$ is symmetric around 0, see the example below.
\begin{proof}
  When combined with a suitable Galilean transformation, condition
  \eqref{EXISTPLS} immediately yields the following existence
  condition for PLS with frequency $\Omega$ and profile $f_\text{s}$
\[
F_r(\Omega)=1\quad\text{where}\quad F_r(\Omega)=\frac1{r}\int_{\R}\beta\left(\frac{\omega+\Omega}{Kr}\right)g(\omega)d\omega.
\]
In order to prove the existence of a solution $(r,\Omega)$ when $f_\text{hom}$ is unstable, notice that
\begin{itemize}
\item[$\bullet$] $F$ is continuous at every $(r,\Omega)\in (0,1]\times \R$, as a consequence of $|\beta(\cdot)|\leq 1$ and Lebesgue dominated convergence.
\item[$\bullet$] ${\displaystyle\lim_{\Omega\to \pm\infty}\sup_{r\in (0,1]}}|F_r(\Omega)|=0$ as a consequence of $F_r(\Omega)=K{\displaystyle\int_{\R}}\beta(\omega)g(Kr\omega-\Omega)d\omega$ and dominated convergence again.
\end{itemize}
Extending $F_r$ by continuity to $\overline{\R}$, the expression
$\{F_r(\Omega)\}_{\Omega\in\overline{\R}}$ defines, for every
$r\in (0,1]$, a closed path in the complex plane. As the next
statement reveals, the limit $r\to 0$ also defines a closed path via
the quantity involved in \eqref{STABHOMOG}.
\begin{Lem}
If $g$ is Lipschitz continuous, then the limit $F_{0+0}(\Omega)$ exists for every $\Omega\in\R$ and we have
\[
F_{0+0}(\Omega)=\frac{K}2\int_{\R^+}\hat{g}(\tau)e^{i\Omega\tau}d\tau,\ \forall \Omega\in\R.
\]
\end{Lem}
The proof is given below. As argued in \cite{D16a,FG-VG16}, continuity in $\Omega$ and the Riemann-Lebesgue lemma ensure that $\{F_{0+0}(\Omega)\}_{\Omega\in\overline{\R}}$ is also a closed path. Moreover, these references showed that, assuming $\|\widehat{g}\|_{L^1_{1+\tau}(\R^+)}<+\infty$, this path winding number around the point $z=1$ is non-zero if \eqref{STABHOMOG} fails for some $z$ with $\text{Re}(z)>0$,

On the other hand, the definition of $\beta$ implies that $\text{Re}(\beta(\omega))\leq 1$ for all $\omega\in\R$, with strict inequality when $\omega\neq 0$. It follows that
\[
\text{Re}(F_1(\Omega))<\int_{\R}g(\omega)d\omega=1,\ \forall \Omega\in\R.
\]
The limits $F_1(\pm\infty)=0$ then imply that
$\{F_1(\Omega)\}_{\Omega\in\overline{\R}}$ winding number around $z=1$
must but 0. By the uniform decay above, there must exist $r\in (0,1)$
for which the path $\{F_r(\Omega)\}_{\Omega\in\overline{\R}}$ contains
$z=1$; hence the rotating PLS. The proof of the Proposition is
complete.

\noindent
{\sl Proof of the Lemma.} We rely on the Plemelj formula
\[
\int_{\R^+}\hat{g}(\tau)e^{i\Omega\tau}d\tau = \pi g(-\Omega)+i\text{ PV}\int_{\R}\frac{g(\omega-\Omega)}{\omega}d\omega
\]
and we separate the integral in $F_r$ into the domain $|\omega|<r^{2/3}$ and $|\omega| \geq r^{2/3}$. In the second domain, we have for small $r$
\[
\frac1{r} \beta(\frac{\omega}{Kr})=\frac{i\omega}{Kr^2}\left(1-\sqrt{1-\left(\frac{Kr}{\omega}\right)^2}\right)=\frac{iK}{2\omega}+O\left(\left(\frac{Kr}{\omega}\right)^2\right)
\]
and then, using that $g\in L^1(\R)$,
\[
\lim_{r\to 0}\frac1{r}\int_{|\omega| \geq r^{2/3}}\beta(\frac{\omega}{Kr})g(\omega-\Omega)d\omega=\frac{iK}{2}\lim_{r\to 0}\int_{|\omega| \geq r^{2/3}}\frac{g(\omega-\Omega)}{\omega}d\omega=\frac{iK}{2}\text{ PV}\int_{\R}\frac{g(\omega-\Omega)}{\omega}d\omega.
\]
In the first domain, we rely on $g$ being Lipschitz continuous to write $g(\omega-\Omega)=g(-\Omega)+\omega h(\omega)$ where $h$ is bounded. Then, we have for $r$ small enough
\[
\frac1{r}\int_{|\omega|< r^{2/3}}\beta(\frac{\omega}{Kr})g(\omega-\Omega)d\omega=\frac{g(-\Omega)}{r}\int_{|\omega|< Kr}\sqrt{1-\left(\frac{\omega}{Kr}\right)^2}d\omega+\frac1{r}\int_{|\omega|< r^{2/3}}\beta(\frac{\omega}{Kr})\omega h(\omega)d\omega.
\]
That $h$ is bounded implies that the second integral vanishes in the limit $r\to 0$. The Lemma then follows from the fact that $\int_{|x|< 1}\sqrt{1-x^2}\ dx=\frac{\pi}2$.
\end{proof}
To conclude this Section, we provide an example of intricate bifurcation diagram (Fig.\ \ref{TRICAUCHY}, right) similar to those reported in the Kuramoto-Sakaguchi model \cite{OW12,OW13}, but obtained in \eqref{KPDE} for the tri-Cauchy distribution (Fig.\ \ref{TRICAUCHY}, left)
\[
g_{\Delta,\Omega,\alpha}=(1-\alpha)g_{1,0}+\alpha g_{\Delta,\Omega}
\]
where $g_{\Delta,\Omega}$ is the bi-Cauchy distribution defined above.
\begin{figure}[ht]
\begin{center}
\includegraphics[scale=0.45]{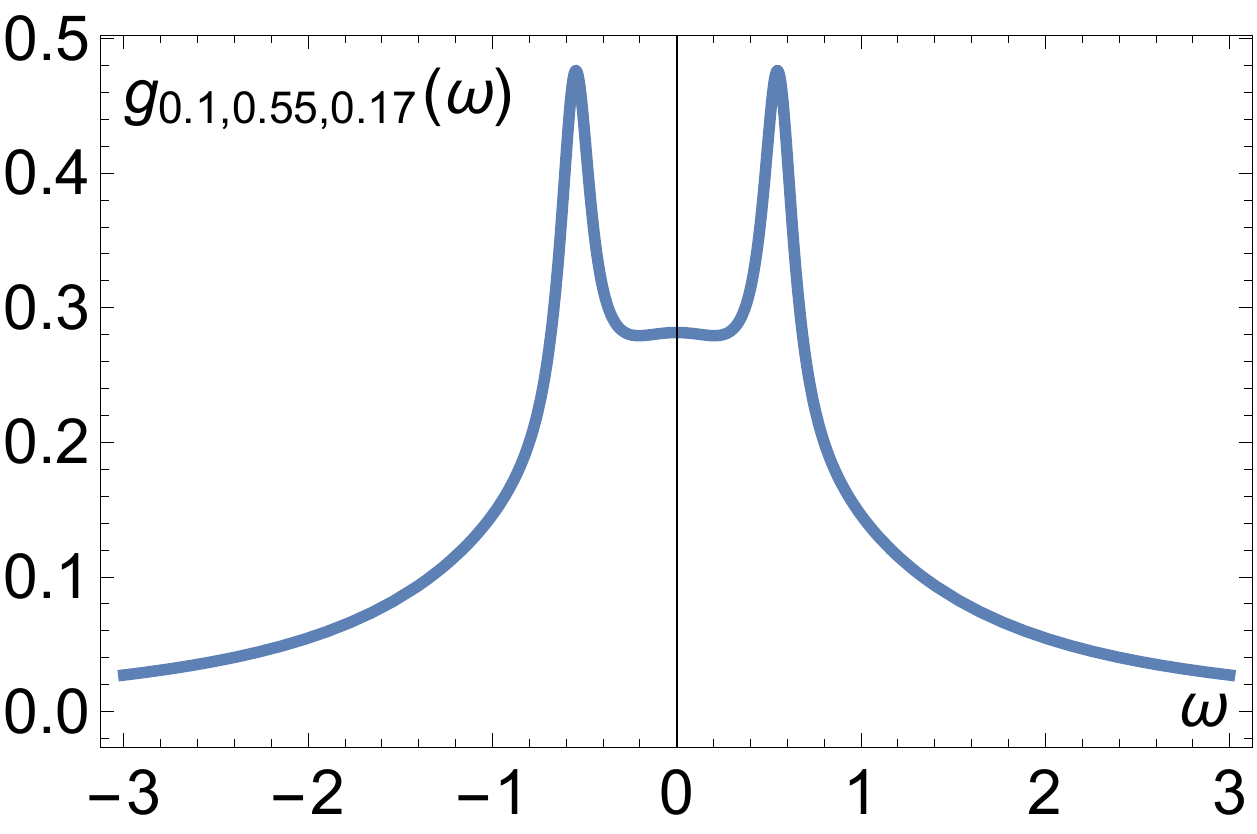}
\hspace{0.5cm}
\includegraphics[scale=0.55]{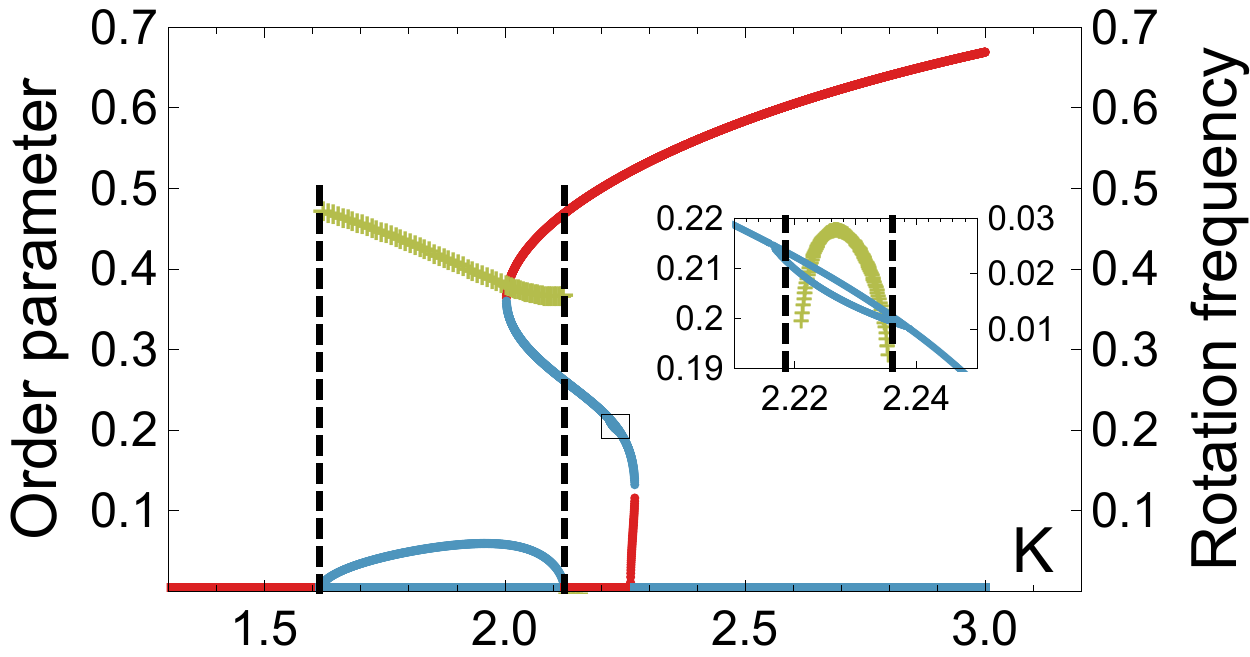}
\end{center}
\caption{{\sl Left.} Graph of the tri-Cauchy distribution for $g_{\Delta,\Omega,\alpha}$ for $\Delta=0.1$, $\Omega=0.55$ and $\alpha=0.17$. {\sl Right.} Corresponding numerically computed bifurcation diagram in \eqref{KPDE}. Red (resp.\ blue) points indicate stable (resp.\ unstable) PLS, or $f_\text{hom}$ when the order parameter is zero. For such $K$ where rotating PLS exist, Green $+$ indicate the global rotation frequency $|\Omega|$, when non-zero. Inset: Zoom into the region $(K,r)\in [2.21,2.25]\times [0.19,0.22]$ where a rotating PLS branch ($\Omega\in [0,0.03]$) emerges from the unstable stationary PLS branch,}
\label{TRICAUCHY}
\end{figure}
In particular, while the second part of the diagram is reminiscent of
the saddle-node bifurcation associated with the bi-Cauchy
distribution, the first destabilization scheme of $f_\text{hom}$ at
$K\simeq 1.61$ is original and generates a branch of rotating PLS
pairs, with frequency $\Omega$ and $-\Omega$ respectively. Also,
$f_\text{hom}$ becomes stable again for $K\simeq 2.12$ and then
suffers a pitchfork bifurcation at $K\simeq 2.27$, as for a unimodal
distribution.

\section{Final comments and open questions}\label{S-EXT}
Proving asymptotic decay in the Kuramoto PDE has essentially consisted
in reducing to a Volterra equation which captures stabilization
mechanisms of the linearized dynamics.  This approach, and the control
of the remaining nonlinear terms, is not limited to the basic
model. It can be shown to extend to various extensions such as when
$f$ also depends on an additional connectivity parameter
$k\in {\mathcal D}$ (${\mathcal D}\subset \R^n$, compact) and the
potential writes
\[
V[f](\theta,\omega,k)=\omega + K \int_{\T^1\times\R\times {\mathcal D}} \alpha(k,k')\sin(\theta' - \theta-\beta) f(d\theta',d\omega',dk'),\ \forall (\theta,\omega,k)\in\T^1\times\R\times {\mathcal D}
\]
where $\beta\in\R$ and $\alpha:{\mathcal D}\times {\mathcal D}\to \R^+$ is assumed to be Lipschitz continuous.

When the connectivity parameter is irrelevant (ie.\ $\alpha\equiv 1$),
the resulting PDE governs the dynamics of empirical measures of the
so-called Kuramoto-Sakaguchi model \cite{SY86}. The very same analysis
as in Section \ref{S-PROOF} can be developed to obtain stability
conditions \cite{OW12,OW13} and prove, without any additional
conceptual obstacle, asymptotic stability of stationary solutions.

Otherwise, when ${\mathcal D}$ is a finite set, the equation describes
the continuum limit of interacting communities of coupled oscillators
\cite{BHOS08,MKB04}. The analysis repeats for the measure vector
$\{f_k(d\theta,d\omega)\}_{k\in {\mathcal D}}$, without any difficulty
other than having to deal with multi-dimensional Volterra equations
for the evolution of the order parameter vector
$\{r_k\}_{k\in {\mathcal D}}$ with components
\[
r_k=\int_{\T^1\times\R}e^{i\theta}f_k(d\theta,d\omega).
\]

More general cases, when ${\mathcal D}$ is infinite, include modelling of networks with random interactions. For arbitrary $\alpha$, explicit existence and stability conditions might be out of reach, especially for PLS. However, when this function decomposes into a product over individual variables \cite{I04,RO14}
\[
\alpha(k,k')=\alpha_1(k)\alpha_2(k')
\]
as when preferential attachment takes place \cite{BA99}, then a self-consistent Volterra equation holds for the integrated order parameter
\[
\int_{\T^1\times\R\times {\mathcal D}}e^{i\theta}\alpha_2(k')f(d\theta,d\omega,dk')
\]
and the analysis entirely repeats in this case.

Finally, here are two problems that remain unsolved, if not unaddressed:
\begin{itemize}
\item Prove asymptotic stability of other remarkable solutions of the Kuramoto PDE \eqref{KPDE}, such as the standing waves discussed in \cite{C94}.
\item Prove asymptotic stability of stationary states (and other remarkable states) in extensions of the Kuramoto model for which interactions include several Fourier modes, as in the so-called Daido model \cite{D96}. The proof in \cite{FG-VG16} (inspired from \cite{FR16}) of Landau damping to $f_\text{hom}$ straightforwardly extends to this case. However, the problem of asymptotic stability of singular states, such as PLS, remains entirely open.
\end{itemize}

\ethics{This work did not involve ethics issue.}

\dataccess{This article has no additional data.}

\aucontribute{All authors have contributed to the paper.}

\competing{We have no competing issue.}

\funding{H.D.\ is supported by Universit\'e Sorbonne Paris Cit\'e, in the framework of the ``Investissements d'Avenir'', convention ANR-11-IDEX-0005, and the People Programme (Marie Curie Actions) of the European Union's Seventh Framework Programme (FP7/2007-2013) under REA grant agreement n.\ PCOFUND-GA-2013-609102, through the PRESTIGE programme coordinated by Campus France.}

\ack{We are grateful to Stanislav M. Mintchev for careful reading of the manuscript, comments and suggestions.}

\disclaimer{No particular disclaimer applies.}

\end{document}